\documentclass[11pt]{amsart}
\usepackage{amsmath,amsthm,amssymb,latexsym,mathrsfs}
\newcommand{\cal}[1]{\mathscr{#1}}
\newcommand{\scr}[1]{\mathcal{#1}}
\newcommand{\fk}[1]{\mathfrak{#1}}
\newcommand{\bC}{\mathbb{C}}
\newcommand{\bZ}{\mathbb{Z}}
\newcommand{\rank}{\mathrm{rank}}

\newtheorem{thm}{Theorem}[section]
\newtheorem*{thm*}{Theorem}
\newtheorem{lem}[thm]{Lemma}

\newtheorem{rem}[thm]{Remark}

\numberwithin{equation}{section}

\usepackage[backend=bibtex,style=ieee]{biblatex}
\bibliography{KZ-image}

\title[Image of KZ Functor]{The image of the KZ functor for Cherednik algebras of varieties with finite group actions}
\author{Daniel Thompson}
\email{dthomp@math.mit.edu}
\date{\today}
\begin{document}
\begin{abstract}
We prove that the KZ functor from a certain category of modules for the Cherednik algebra to finite dimensional modules over the Hecke algebra is essentially surjective.  Then we begin to use this result to study the analog of category $\scr{O}$ for Cherednik algebras on Riemann surfaces and on products of elliptic curves.  In particular we give conditions on the parameters under which these categories are nonzero.
\end{abstract}
\maketitle
\section{Introduction}
Let $W$ be a finite complex reflection group inside of $GL(V)$.  The paper \cite{BMR} has defined a Hecke algebra $H$ associated to $W$, which is a deformation of the group algebra $\bC W$.  An important open question is whether this deformation is flat; in particular is it finite dimensional for any choice of the deformation parameters.  Recent work of I. Losev, \cite{Lo14}, has made progress on this question using the theory of category $\scr{O}$ for rational Cherednik algebras.  Namely, he has shown that the KZ functor from $\scr{O}$ to the category of finite-dimensional representations of $H$ is essentially surjective, which implies that there is a minimal two-sided ideal of finite codimension in $H$.  Let us note also that category $\scr{O}$ is the category of modules over the Cherednik algebra which are finitely generated over the subalgebra $\bC[V]$ and for which every irreducible composition factor has regular singularities at $\infty$ when viewed as a $D$-module on a dense open set of its support (Proposition 1.3 of \cite{Wi11}).

There is a more general theory of Cherednik and Hecke algebras for varieties with a finite group action introduced by P. Etingof in a preprint \cite{Et04}.  It is natural to ask whether the KZ functor in this setting is surjective as well.  This question in answered in the affirmative by
\begin{thm} \label{thm:main}
  Let $V$ be a finite-dimensional representation of the Hecke algebra of a variety with a finite group action.  There is an $\cal{O}$-coherent sheaf of modules $M$ over the sheaf of Cherednik algebras such that $KZ(M)=V$.
\end{thm}
If our variety is proper, then any such $M$ has regular singularities in the sense above.  We suspect that for any variety we can choose some such $M$ which has regular singularities, and plan to return to this question in the future.

The structure of the paper is as follows.  First, we briefly recall the basic theory we will need of global Cherednik algebras in section 2.  In section 3 we extend the construction of Losev to the setting of Cherednik algebras of varieties, proving Theorem \ref{thm:main}.  In sections 4-6 we begin to use these results to study the analog of category $\scr{O}$ for Cherednik algebras on Riemann surfaces and on products of elliptic curves.  In particular we give conditions on the parameters under which these categories are nonzero.

\subsection*{Acknowledgements}
The author would like to thank Pavel Etingof and Seth Shelley-Abrahamson for their helpful comments on preliminary versions of this paper.  Thanks are due especially to Etingof for suggesting this topic and for his patience and numerous useful discussions.  This material is based upon work supported by the National Science Foundation Graduate Research Fellowship Program under Grant No. 1122374.

\section{Generalities}
\subsection{Cherednik algebras of varieties with a finite group action}
Let $X$ be a smooth connected algebraic variety over $\bC$, with a finite group $W$ of automorphisms.  Assume further that $X$ admits a cover by affine $W$-invariant open sets, so that the quotient variety $X/W$ exists.  Let $c$ be a $W$-invariant $\bC$-valued function on the set $S$ of pairs $(Z,s)$, where $s\in W$ and $Z$ is an irreducible component of the set of fixed points $X^s$ of codimension 1 in $X$.  These elements $s\in W$ will be called \emph{reflections}.  Finally, let $\eta$ be a $W$-invariant $\bC$-valued function on the set of hypersurfaces $Z$ such that $(Z,s)\in S$ for some $s\in W$.  Etingof in \cite{Et04} has defined a sheaf of algebras $H_{c,\eta}(W,X)$ on $X/W$ called the \emph{Cherednik algebra} associated to this data.  We will consider the full subcategory $H_{c,\eta}-\mathrm{mod}_{\mathrm{coh}}$ of $H_{c,\eta}-\mathrm{mod}$, consisting of sheaves of modules which are coherent as $\cal{O}_X$-modules (in the sequel, an $\cal{O}$-\emph{coherent} module over the Cherednik algebra will always mean coherent over $\cal{O}_X$).  Our main object of interest will be a certain Serre subcategory $H_{c,\eta}-\mathrm{mod}_{\mathrm{RS}}$ of $H_{c,\eta}-\mathrm{mod}_{\mathrm{coh}}$, see Definition 1.5 of \cite{Wi11}.  These are $\cal{O}$-coherent sheaves of modules for which each irreducible composition factor has regular singularities considered as an equivariant local system on a dense open set of its support in $X$.  Note we are employing what Etingof calls the ``modified Cherednik algebra'' $H_{1,c,\eta,0,X,W}$, defined in Section 2.10 of \cite{Et04}.

For convenience we recall the definition from \cite{Et04} of the \emph{Hecke algebra} $\cal{H}_\tau(W,X)$ of $X,W$, which is a deformation of the group algebra of the orbifold fundamental group.  Let $X^{\mathrm{reg}}$ be the set of points of $X$ with trivial stabilizer, an open set.  For $(Z,s)\in S$, let $n_Z$ be the order of the pointwise stabilizer of $Z$ in $G$.  Choosing a basepoint $x_0\in X^{\mathrm{reg}}$, the braid group of $X/W$ is defined to be the fundamental group $\pi_1(X^{\mathrm{reg}}/W,x_0)$.  Let $C_Z$ be the conjugacy class in $\pi_1(X^{\mathrm{reg}}/W,x_0)$ represented by a small circle going counterclockwise around the image of $Z$ in $X/W$.  Now for any conjugacy class of hypersurfaces $Z$ we introduce parameters $\tau_{1Z}, \ldots, \tau_{n_Z Z} \in \bC$.  The Hecke algebra is defined to be the quotient of the group algebra of the braid group by the relations
$$\prod_{j=1}^{n_Z} (T-e^{2\pi i j/n_Z} e^{\tau_{j Z}}) = 0, \quad T\in C_Z.$$

Let us now recall the Knizhnik-Zamolodchikov functor.  On the complement $X'$ of all the hypersurfaces $Z$, the sheaf $H_{c,\eta}$ localizes to $\bC[W] \ltimes D_{X'}$.  In particular, the restriction to $X^{\mathrm{reg}}/W$ of a module $M\in H_{c,\eta}-\mathrm{mod}_{\mathrm{coh}}$ is a $W$-equivariant local system on $X^{\mathrm{reg}}$.  Taking monodromy gives a functor to the category of finite-dimensional representations of the braid group $\pi_1(X^{\mathrm{reg}}/W,x_0)$.  Etingof shows in Proposition 3.4 of \cite{Et04} that the composition of restriction and this equivalence lands in $\cal{H}_\tau-\mathrm{mod}$, the category of finite-dimensional representations of the Hecke algebra $\cal{H}_{\tau(c,\eta)}(W,X)$, where
$$\tau_{j Z}=\frac{2\pi i}{n_Z}\left( \sum_{m=1}^{n_Z-1} 2c(Z,s_Z^m) \frac{1-e^{2\pi i j m/n_Z}}{1-e^{-2\pi i m/n_Z}} +\eta(Z)\right).$$
We denote this functor as 
$$KZ: H_{c,\eta}-\mathrm{mod}_{\mathrm{coh}} \to \cal{H}_\tau-\mathrm{mod},$$
just as we do for its restriction to $H_{c,\eta}-\mathrm{mod}_{\mathrm{RS}}$,
$$KZ: H_{c,\eta}-\mathrm{mod}_{\mathrm{RS}} \to \cal{H}_\tau-\mathrm{mod}.$$
These functors are both exact.

\subsection{Isomorphisms of \'etale lifts}
If $\zeta:U\to X/W$ is an \'etale map, then the pullback $\zeta^* H_{c,\eta}(W,X)$ has a natural structure of a sheaf of algebras on $U$, which is affine-locally generated by $\zeta^* \cal{O}_X$, $\bC[W]$, and the Dunkl operators. S. Wilcox \cite{Wi11}, in Proposition 2.3, shows that if now $\tilde{\zeta}:\tilde{U} \to X$ is a $W$-equivariant \'etale map, then the image of each component $Z$ of $\tilde{U}^s$ of codimension 1 is a component of $X^s$ of codimension 1.  Writing $\zeta:\tilde{U}/W \to X/W$, he shows that $\zeta^* H_{c,\eta}(W,X)$ is naturally isomorphic to $H_{c,\eta}(W,\tilde{U})$ where by abuse of notation $c$ (resp. $\eta$) is the appropriate restriction of $c$ (resp. $\eta$).

Similarly if $U$ is a Stein complex analytic manifold with an \'etale map $U\to X/W$, then $\bC_{an}[U]\otimes_{\cal{O}_{X/W}} H_{c,\eta}(W,X)$ is naturally an algebra.

Let $(Z,s)\in S$ and let $W'=W_Z$ be the subgroup of elements of $W$ which act trivially on $Z$.  Now $W'$ is cyclic and a generic point of $Z$ has $W'$ as its stabilizer.  Set $X^{\mathrm{reg}-W'} = \{x\in X\, |\, W_x \subset W'\}$.  This is an open set of $X$, and the natural morphism $\zeta_Z: X^{\mathrm{reg}-W'}/W' \to X/W$ is \'etale.  Let us define the centralizer algebra, $Z(W,W',H_{c,\eta}(W',X^{\mathrm{reg}-W'}))$.  This is a sheaf of algebras on $X^{\mathrm{reg}-W'}/W'$ defined by $\cal{E}nd_{H_{c,\eta}(W',X^{\mathrm{reg}-W'})^{op}}(P)$ where $P$ is the right $H_{c,\eta}(W',X^{\mathrm{reg}-W'})$-module $\cal{F}un_{W'}(W,H_{c,\eta}(W',X^{\mathrm{reg}-W'}))$ of $W'$ invariant $H_{c,\eta}(W',X^{\mathrm{reg}-W'})$-valued functions on $W$.  There is an element $e(W')\in \Gamma(X^{\mathrm{reg}-W'}/W',Z(W,W',H_{c,\eta}(W',X^{\mathrm{reg}-W'}))$ such that for all $f\in \Gamma(U,P)$, $e(W')f(u) = f(u)$ if $u\in W'$ and 0 otherwise.  Then $e(W')Z(W,W',H_{c,\eta}(W',X^{\mathrm{reg}-W'}))e(W')$ is naturally identified with \linebreak $H_{c,\eta}(W',X^{\mathrm{reg}-W'})$.

In parallel to Lemma 2.1 of \cite{Lo14}, we have the following lemma.
\begin{lem} \label{isom}
There is a natural isomorphism of $\cal{O}_{X^{\mathrm{reg}-W'}/W'}$-quasicoherent sheaves of algebras on $X^{\mathrm{reg}-W'}/W'$
$$\theta: \zeta_Z^* H_{c,\eta}(W,X) \to Z(W,W',H_{c,\eta}(W',X^{\mathrm{reg}-W'})).$$
Moreover, we have a Morita equivalence $$Z(W,W',H_{c,\eta}(W',X^{\mathrm{reg}-W'}))-\mathrm{mod} \to H_{c,\eta}(W',X^{\mathrm{reg}-W'})-\mathrm{mod}$$ given by $M\mapsto e(W')M$.
\end{lem}
\begin{proof}
  To prove that this map is an isomorphism it suffices to check that it is when restricted to formal neighborhoods of closed points of $X^{\mathrm{reg}-W'}/W'$.  Now the result follows from Proposition 2.6 of \cite{Wi11}.
\end{proof}

\section{The KZ functor is essentially surjective}
\subsection{}
Let $V$ be a finite-dimensional $\cal{H}_\tau$-module, and let $N$ denote the corresponding $W$-equivariant local system on $X^{\mathrm{reg}}$ with regular singularities.  Following Losev \cite{Lo14} we will produce a module $M\in H_{c,\eta}-\mathrm{mod}_{\mathrm{coh}}$ whose restriction to $X^{\mathrm{reg}}/W$ is isomorphic to $N$.

Let $(Z,s)\in S$ and let $W' = W_Z$ be the subgroup of elements of $W$ which act trivially on $Z$.  Using Lemma \ref{isom}, we may view $e(W')\zeta_Z^* N$ as a module over $H_{c,\eta}(W',X^{\mathrm{reg}-W'})$.

\subsection{Extension to codimension 1}
Choose $U$ a $W'$-stable complex analytic neighborhood in $X^{\mathrm{reg}-W'}$ which intersects $Z$ but no other component of $X^{W'}$.  The action of $W'$ on $U$ is linear in some chart $\fk{h}$ (a complex vector space), so we may take an analytic disk $Y\subset Z$, a unique $W'$-stable complement $\fk{h}_{W'}$ to $\fk{h}^{W'}$, and a disk $D\subset \fk{h}_{W'}$ so that $\hat{Y}:= Y\times D \subset U$.  Of course, $\hat{Y}$, as well as $\hat{Y}^\times := \hat{Y} \setminus Y$ are Stein so we may consider \'etale pullbacks to these neighborhoods.

Note that we have an isomorphism
\begin{equation} \label{eq:isom}
  \bC_{an}[\hat{Y}/W'] \otimes_{\cal{O}_{X^{\mathrm{reg}-W'}/W'}} H_{c,\eta}(W',X^{\mathrm{reg}-W'}) \simeq \bC_{an}[\hat{Y}/W'] \otimes_{\bC[\fk{h}/W']} H_{c,\eta}(W',\fk{h}),
\end{equation}
where $H_{c,\eta}(W',\fk{h})$ is the (modified) rational Cherednik algebra.
Moreover, the second algebra coincides with
$$\bC_{an}[\hat{Y}/W'] \otimes_{\bC[\fk{h}_{W'}/W'] \otimes \bC[Y]} (H_{c,\eta}(W',\fk{h}_{W'}) \otimes D_Y).$$

Set $N_Z = e(W')(\bC_{an}[\hat{Y}^\times/W']\otimes_{\cal{O}_{X^{\mathrm{reg}}/W}} N)$.  This is a $W'$-equivariant meromorphic local system on $\hat{Y}^\times$ with regular singularities on $Y$.  Let $M_Z$ be the $\bC_{an}[\hat{Y}]$-span of the set of meromorphic sections of $N_Z$ which are annihilated by vector fields of $Y$ and which lie in the $0$-generalized eigenspace for the action of $\fk{h}_{W'}$ under the isomorphism (\ref{eq:isom}) above.  Now $M_Z$ is a module over $\bC_{an}[\hat{Y}/W']\otimes_{\cal{O}_{X^{\mathrm{reg}-W'}/W'}} H_{c,\eta}(W',X^{\mathrm{reg}-W'})$ which is coherent over $\bC_{an}[\hat{Y}]$.  Moreover, $\bC_{an}[\hat{Y}^\times] \otimes_{\bC_{an}[\hat{Y}/W']} M_Z = N_Z$, as can be checked fiberwise using the fact that the KZ functor is essentially surjective for rational Cherednik algebras of cyclic groups (this is true by Theorem 5.14 of \cite{GGOR} because the Hecke algebra is finite dimensional in this case).

Now we let $\tilde{N}_Z = M_Z \cap e(W')\zeta_Z^* N,$ the intersection taken inside of $N_Z$.  Viewing $Z$ inside of $X^{\mathrm{reg}-W'}/W'$ and covering $Z$ with open sets, we see that this definition glues to a well-defined sheaf on $X^{\mathrm{reg}-W'}/W'$.  This is a $H_{c,\eta}(W',X^{\mathrm{reg}-W'})$-submodule of $e(W')\zeta_Z^* N$.
\begin{lem}
The sheaf of modules $\tilde{N}_Z$ is $\cal{O}$-coherent away from the set of components of $X^{W'}$ other than $Z$, and it restricts to $e(W')\zeta_Z^* N$ away from $Z$.
\end{lem}
\begin{proof}
Apply the proof of Lemma 3.2 of \cite{Lo14} to an affine open cover of $\left( X^{\mathrm{reg}-W'}/W' - X^{W'} \right) \cup Z$.
\end{proof}

Now exactly as in \cite{Lo14}, we are able to produce from the $\tilde{N}_Z$ a module $\tilde{N} \in H_{c,\eta}(W,X^{\mathrm{sr}})-\mathrm{mod}_{\mathrm{coh}}$, where $X^{\mathrm{sr}} = X^{\mathrm{reg}} \cup \bigcup_{(Z,s)\in S} (Z \cap X^{\mathrm{reg}-W_Z})$ is the subregular locus.  Let $\iota_Z:X^{\mathrm{reg}-W_Z}\to X^{\mathrm{sr}}$ and $\iota:X^{\mathrm{reg}}\to X^{\mathrm{sr}}$ be the inclusions and $\pi_Z:X^{\mathrm{reg}-W_Z} \to X^{\mathrm{reg}-W_Z}/W_Z$ and $\pi:X \to X/W$ the quotient maps.  Write $\tilde{N}_Z$ to refer, by abuse of notation, to the corresponding $\zeta_Z^* H_{c,\eta}(W,X)$-module under the isomorphism of Lemma \ref{isom}, and note that $\iota_{Z*} \pi_Z^* \tilde{N}_Z \subset \iota_* \pi^* N$.  Put $\hat{N} = \bigcap_Z \iota_{Z*} \pi_Z^* \tilde{N}_Z$, an $\cal{O}$-coherent sheaf on $X^{\mathrm{sr}}$.  Finally set $\tilde{N} = \pi_*(\hat{N})^W$.

\begin{proof}[Completion of the proof of Theorem \ref{thm:main}]
  The complement of $X^{\mathrm{sr}}$ in $X$ has codimension at least 2, so by Lemma 3.6 of \cite{Wi11}, there is an $M \in H_{c,\eta}(W,X)-\mathrm{mod}_{\mathrm{coh}}$ which restricts to $\tilde{N}$.  On an open affine set $U$ of $X/W$, we have $\Gamma(U,M) = \Gamma(U,\tilde{N})$.
\end{proof}

\begin{rem} \label{rem:RS}
We expect that we may take $M$ to be in the subcategory $H_{c,\eta}-\mathrm{mod}_{\mathrm{RS}}$.  This is easily seen to be true if $X$ is a curve, if $X$ is a proper variety of any dimension (since in this case $H_{c,\eta}-\mathrm{mod}_{\mathrm{coh}}=H_{c,\eta}-\mathrm{mod}_{\mathrm{RS}}$), or if $V$ is irreducible (since $KZ$ is exact).  The difficulty in the noncompact case arises from the fact that for some $V$, any $M$ will have composition factors that are supported on high codimension subvarieties, but are sandwiched between factors with full support.
\end{rem}

\section{Category $\scr{O}$ for the Cherednik algebra of a complex curve with finite group action}
\subsection{}
Let $X$ be a compact Riemann surface with a finite subgroup $W\subset Aut(X)$, and let $\Sigma = X/W$, regarded as an orbifold having special points $P_i, i=1,\ldots,m$ with stabilizers $\bZ_{n_i}$.  Let $g$ be the genus of $\Sigma$.

We would like to give necessary and sufficient conditions on the parameters $c = (c_{ij})_{i=1,\ldots, m; j=1,\ldots,n_i-1}, \eta$ so that the category $\scr{O}=H_{c,\eta}-\mathrm{mod}_{\mathrm{RS}}$ of $\cal{O}$-coherent modules over the sheaf of Cherednik algebras $H_{c,\eta}(W,X)$ is nonzero (according to Remark \ref{rem:RS} such modules have regular singularatites in the sense above since $X$ is compact).

The set-theoretic support of an irreducible module in $\scr{O}$ is either a single point $P_i$, or it is all of $X/W$.  By \cite{Et04} Proposition 2.22, there exists a module in $\scr{O}$ set-theoretically supported at $P_i$ if and only if there exist positive integers $a,b$ satisfying
\begin{equation} \label{eq:ptsupp}
  a = 2 \sum_{j=1}^{n_i-1} \frac{1-\zeta_i^{ja}}{1-\zeta_i^j} \zeta_i^{jb} c_{ij},
\end{equation}
where $\zeta_i$ is a primitive ${n_i}^\mathrm{th}$ root of unity.

Recall that we make the invertible linear transformation $(c,\eta)\mapsto \tau(c,\eta)$ by the formula
$$\tau_{ij} := \tau_{j P_i} = 2\pi \sqrt{-1} \left(2\sum_{k=1}^{n_i-1} c_{ij} \frac{1-\zeta_i^{jk}}{1-\zeta_i^{-k}} - \eta_i \right) /n_i.$$

Recall also that the Hecke algebra $\cal{H}_\tau(W,X)$ is generated over $\bC[\tau]$ by generators $T_i, i=1,\ldots,m; A_l, B_l, l=1,\ldots, g$ with defining relations
\begin{equation} \label{eq:eig}
  \prod_{j=1}^{n_i} (T_i-\zeta_i^j exp(\tau_{ij})) = 0,
\end{equation}
\begin{equation}
  \quad T_1 T_2 \cdots T_m = \prod_{l=1}^g A_l B_l A_l^{-1} B_l^{-1}.
\end{equation}
The functor $KZ: \scr{O} \to \cal{H}_\tau-\mathrm{mod}$ is essentially surjective by Theorem \ref{thm:main}, and its kernel is the Serre subcategory generated by irreducible modules supported at one of the points $P_i$.  Having already described when the kernel is nonzero, we have only to find conditions under which $\cal{H}_\tau$ has finite dimensional modules.

Suppose we have a representation of $\cal{H}_\tau$ given by $T_i, A_l, B_l \in GL_d$.  For each $i$, let 
\begin{equation} \label{eq:conj}
\alpha_{ij} = \rank((T_i - \zeta_i exp(\tau_{i1})I)\cdots(T_i-\zeta_i^j exp(\tau_{ij})I))
\end{equation}
for $j=1, \ldots, n_i$, $\alpha_{i0}=d$ for all $i$.  Then we must have
\begin{equation} \label{eq:ds}
1=\det(T_1)\cdots\det(T_m) = \prod_{i=1}^m \prod_{j=1}^{n_i} (\zeta_i^j exp(\tau_{ij}))^{\alpha_{i,j-1}-\alpha_{ij}}.
\end{equation}
In fact, if $g\geq 1$ then by a theorem of Shoda (see \cite{PW}), (\ref{eq:ds}) is the only obstruction to finding a representation with $T_i$ in the conjugacy classes defined by (\ref{eq:conj}).

We summarize this below.
\begin{thm}
  Suppose $X$ is compact and $g(X/W) \geq 1$.  The category $\scr{O}_{c,\eta}(W,X)$ is nonzero if and only if the parameters $(c,\eta)$ lie on one of the hyperplanes in the countable collection defined by (\ref{eq:ptsupp}) or (\ref{eq:ds}).
\end{thm}

\begin{rem}
  It is easy to see that if $X/W$ is not compact then $\scr{O}$ is nonzero.  This is because the Hecke algebra is generated by elements $T_i,A_l,B_l$ as above, as well as elements $X_p$ corresponding to loops around points at infinity.  Here is one representation: let $T_i$ be any matrices which satisfy the appropriate polynomial relations (\ref{eq:eig}), and let $A_l,B_l$ be identity matrices.  Then we may choose the $X_p$ so that the product of the $T_i$ and the $X_p$ (in the correct order) is $I$.
\end{rem}
  
\subsection{Deligne-Simpson problem}
When $g=0$, the question of the existence of finite dimensional $\cal{H}_\tau$ modules is closely related to the Deligne-Simpson problem \cite{C-B}.  The latter asks when there exist solutions to $T_1 \cdots T_m = I, T_i\in C_i$ where $C_i\subset GL_n$ are the conjugacy classes defined by the equations 
\begin{equation} \label{eq:class}
\alpha_{ij} = \rank((T_i - \xi_{i1}I)\cdots(T_i-\xi_{ij}I))
\end{equation}
for $j=1, \ldots, n_i$.

We turn our attention now to the case when the universal cover $H$ of $X$ is the Euclidean or Lobachevsky plane.  If $W$ contains reflections, then $\cal{H}_\tau(W,X)$ is a GDAHA of rank 1, studied in \cite{EOR} and \cite{ER}.  In the Euclidean case there are four possibilities for $W$:
$$m=4, (n_1, n_2, n_3, n_4) = (2,2,2,2) \quad (\textrm{case } Q=D_4^{(1)}),$$
which yields Cherednik's DAHA of type $C^\vee C_1$, and
$$m=3, (n_1, n_2, n_3) = (3,3,3), (2,4,4), (2,3,6) \quad (\textrm{cases } Q=E_6^{(1)}, E_7^{(1)}, E_8^{(1)}).$$
Here $W$ is cyclic of order 2,3,4, or 6 respectively.

In general, let $Q$ be a star-shaped graph having $m$ legs coming out of the central node, of lengths $n_1-1, \ldots, n_m-1$.  $Q$ will be finite if $H$ is the Riemann sphere, affine if $H$ is the plane, and indefinite if $H$ is the hyperbolic plane.  To avoid redundant parameters we may take $\eta(P_1)=\eta, \eta(P_i)=0$ for $i\neq 1$.  

In the $D_4^{(1)}$ case we may write $t_i = exp(-\pi \sqrt{-1} c_i)$, $$q=exp(-\pi \sqrt{-1} (c_1 + c_2 + c_3 + c_4 - \eta))$$ and present $\cal{H}_\tau(W,X)$ as:
$$(T_i - t_i)(T_i + t_i^{-1}) = 0, \quad T_1T_2T_3T_4 = q.$$
More generally, for any GDAHA set
$$q =\prod_{i,j} \left(\zeta_i^j exp(\tau_{ij})\right)^{-1/n_i}.$$
It is shown in Section 5 of \cite{EOR} that if $Q$ is affine Dynkin and $q$ is a root of unity then $\scr{O}_{c,\eta}(W,X)$ is nonzero.  In type $D_4^{(1)}$, for $q$ not a root of unity \cite{OS} gives a precise description of the set of parameters for which $\scr{O}$ is nonzero.  We may do the same for a GDAHA of arbitrary type.

For this, we must view a tuple of integers $\alpha = (\alpha_{ij})$ as an element of the root lattice for the root system of type $Q$, where $\alpha_{ij}$ is the coefficient of the simple root corresponding to the $j^{\mathrm{th}}$ vertex on the $i^\mathrm{th}$ leg of the affine/indefinite Dynkin diagram $Q$.  For all $i$ we let $\alpha_{i0} = \alpha_0$ be the coefficient of the center vertex, and $\alpha_{i,n_i} = 0$.  Let us say that a root $\alpha$ is \emph{strict} if $\alpha_0>0$.  Let $\xi_{ij} = \zeta_i^j exp(\tau_{ij})$, and write $\xi^{[\alpha]} = \prod_{i=1}^m \prod_{j=1}^{n_i} \xi_{ij}^{\alpha_{i,j-1}-\alpha_{ij}}.$

Let us recall Theorem 1.3 of \cite{C-B} on the Deligne-Simpson problem.  Let the conjugacy classes $C_i$ and $\alpha$ be defined as in (\ref{eq:class}) and as before let $Q$ the star-shaped graph having $m$ legs coming out of the central node, of lengths $n_1-1, \ldots, n_m-1$.
\begin{thm}[Crawley-Boevey]
There is a solution to $T_1 \cdots T_m = I$ with $T_i \in \overline{C_i}$ if and only if $\alpha$ can be written as a sum of positive roots for $Q$, say $\alpha = \beta + \gamma + \ldots$ with $\xi^{[\beta]}=\xi^{[\gamma]}=\cdots = 1$.
\end{thm}

This theorem immediately gives us
\begin{thm}
The category $\scr{O}_{c,\eta}(W,X)$ is nonzero if and only if $\xi^{[\alpha]}=1$ for some strict positive root $\alpha$, or if the parameters $(c,\eta)$ lie on one of the hyperplanes in the countable collection defined by (\ref{eq:ptsupp}).
\end{thm}

\section{GDAHA of higher rank: the affine case}
Now let $X$ be an elliptic curve and $W$ a finite group of automorphsims of $X$ which contains reflections, corresponding to the quiver $Q=(Q_0,Q_1)$ as in the previous section.  Then the product $X^n$ carries the action of the wreath product $G=S_n \ltimes W^n$.  We will produce necessary conditions under which $\scr{O}_{c,\eta}(G,X^n)$ is nonzero.

The Hecke algebra $\cal{H}_\tau(G,X^n)$ is the GDAHA of rank $n$, introduced in \cite{EGO}.  As above we write $\tau_{ij} = \tau_{jZ}$ when $Z=P_i \times X^{n-1}$.  The same paper also introduces rational degenerations of these algebras, denoted $B_{n,\mu, \nu}$, which are related to symplectic reflection algebras for wreath products as follows.  Let $\Gamma$ be the finite subgroup of $SL_2(\mathbb{C})$ associated to $Q$ under the McKay correspondence and put $\mathbf{\Gamma}_n = S_n \ltimes \Gamma^n$ for the wreath product group.  Consider the SRA $H(\mathbf{\Gamma}_n)$, let $e\in \bC[\mathbf{\Gamma}_n]$ be the idempotent for the representation corresponding the nodal vertex of $Q$ under the McKay correspondence.  It is shown in Corollary 2.3.6 of \cite{EGO} that $e^{\otimes n}H(\mathbf{\Gamma}_n)e^{\otimes n} = B_{n,\mu,\nu}$.  Note also that $H(\mathbf{\Gamma}_n)$ is Morita equivalent to the Gan-Ginzburg algebra $A_{n,\mu, \nu}$, with an invertible transformation between the parameters for $H(\mathbf{\Gamma}_n)$ and $(\mu, \nu)$ explicitely given, for example, in \cite{E12} Section 2.2.

In what follows, we restrict to the case that the parameters $(\mu,\nu)$ are \emph{spherical}, in the sense that both the idempotent $e^{\otimes n}$ and the trivial (i.e., averaging) idempotent $p$ for $\mathbf{\Gamma}_n$ are spherical: $H(\mathbf{\Gamma}_n) = H(\mathbf{\Gamma}_n) e^{\otimes n} H(\mathbf{\Gamma}_n)$ and $H(\mathbf{\Gamma}_n) = H(\mathbf{\Gamma}_n) p H(\mathbf{\Gamma}_n)$.  According to Conjecture 5.1 of \cite{E12}, the spherical locus contains the complement of an explicit countable set of hyperplanes in the parameter space $\bC^{|Q_0| + 1}$.

Let $\fk{g}$ be the affine Kac-Moody algebra of type $Q$, $\omega_0$ the fundamental weight corresponding to the extending vertex, and $\delta$ the smallest positive imaginary root (see \cite{K} for details).  Let $L_{\omega_0}$ be the irreducible representation of $\fk{g}$ of highest weight $\omega_0$.  Consider the subalgebra $\fk{a}\subset \fk{g}$ generated by the Cartan $\fk{h}\subset \fk{g}$ and all root subspaces $\fk{g}_\beta$ for real roots $\beta = \sum_{i\in Q_0} b_i \alpha_i$ with $\sum_{i\in Q_0} b_i \mu_i \in \mathbb{Z}$.  Finally, let $L^{\fk{a}}_{\omega_0}$ be the $\fk{a}$-submodule of $L_{\omega_0}$ generated by the weight spaces $L_{\omega_0}[\sigma \omega_0]$ for $\sigma\in W$, the Weyl group of $\fk{g}$.

\begin{thm} Suppose that the parameters $(\mu,\nu)$ are spherical and related to $\tau = \tau(c,\eta)$ under the equations
  $$e^{\tau_{ij}} = exp(2\pi i (\gamma_{ij} - j/n_i)), \quad exp(\tau_{1 (X^n)^s})= - e^{-2\pi i \nu}$$ where $s\in S_n \subset G$ is a transposition,
  and where $$\gamma_{ij} = \sum_{p=1}^{j-1} \mu_{ip} + \mu_0/m + \xi_i, \quad \xi_1 + \cdots + \xi_m = 0.$$  Then $\scr{O}_{c,\eta}(G,X^n)$ is nonzero if the weight space $L_{\omega_0}^{\fk{a}} \cap L_{\omega_0}[\omega_0 - n \delta]$ is nonzero. \end{thm}
  \begin{proof}
    If $L_{\omega_0}^{\fk{a}}[\omega_0 - n \delta]\neq 0$ then from Theorem 1.2 of \cite{BL} there is a finite dimensional representation $V$ of the SRA $H(\mathbf{\Gamma}_n)$.  Then $e^{\otimes n} V$ is a finite dimensional representation of $B_n$.  Now we apply the functor of Proposition 4.2.2 of \cite{EGO}.
  \end{proof}
  
\printbibliography
\end{document}